\documentclass[11pt]{article}

\usepackage[T1]{fontenc}
\usepackage{lmodern}
\usepackage[margin=1in]{geometry}
\usepackage{microtype}
\usepackage{mathtools,amssymb,amsthm}
\usepackage{aliascnt}
\usepackage[hidelinks]{hyperref}
\usepackage[nameinlink,noabbrev]{cleveref}
\usepackage{xcolor}
\usepackage{authblk}

\allowdisplaybreaks
\numberwithin{equation}{section}

\newtheorem{theorem}{Theorem}[section]
\newaliascnt{lemma}{theorem}
\newtheorem{lemma}[lemma]{Lemma}
\aliascntresetthe{lemma}
\newaliascnt{proposition}{theorem}
\newtheorem{proposition}[proposition]{Proposition}
\aliascntresetthe{proposition}
\newaliascnt{corollary}{theorem}
\newtheorem{corollary}[corollary]{Corollary}
\aliascntresetthe{corollary}

\theoremstyle{definition}
\newaliascnt{definition}{theorem}
\newtheorem{definition}[definition]{Definition}
\aliascntresetthe{definition}

\theoremstyle{remark}
\newaliascnt{remark}{theorem}
\newtheorem{remark}[remark]{Remark}
\aliascntresetthe{remark}

\crefname{theorem}{theorem}{theorems}
\Crefname{theorem}{Theorem}{Theorems}
\crefname{lemma}{lemma}{lemmas}
\Crefname{lemma}{Lemma}{Lemmas}
\crefname{proposition}{proposition}{propositions}
\Crefname{proposition}{Proposition}{Propositions}
\crefname{corollary}{corollary}{corollaries}
\Crefname{corollary}{Corollary}{Corollaries}
\crefname{definition}{definition}{definitions}
\Crefname{definition}{Definition}{Definitions}
\crefname{remark}{remark}{remarks}
\Crefname{remark}{Remark}{Remarks}
\crefname{example}{example}{examples}
\Crefname{example}{Example}{Examples}

\newcommand{\Q}{\mathbb{Q}}
\newcommand{\R}{\mathbb{R}}
\newcommand{\C}{\mathbb{C}}
\newcommand{\Z}{\mathbb{Z}}
\newcommand{\lc}{\operatorname{lc}}
\newcommand{\disc}{\operatorname{disc}}

\title{Four-square polynomial triples:
quadratic classification and equal-degree obstructions}
\author[1]{Mirela Juki\'{c} Bokun}
\author[2]{Ana Jurasi\'{c}}

\affil[1]{School of Applied Mathematics and Informatics\\
  University of Osijek\\
Trg Ljudevita Gaja 6, HR-31000 Osijek, Croatia\\
\texttt{mirela@mathos.hr}}
\affil[2]{Faculty of Mathematics, University of Rijeka\\
Radmile Matej\v{c}i\'c 2, HR-51000 Rijeka, Croatia\\
\texttt{ajurasic@math.uniri.hr}}
\date{}

\hypersetup{
  pdftitle={Four-square polynomial triples:
quadratic classification and equal-degree obstructions},
  pdfauthor={Mirela Jukic Bokun and Ana Jurasic}
}

\begin{document}

\maketitle

\begin{abstract}
 We study triples of distinct nonconstant polynomials $a,b,c$ over
\(\mathbb{R}\), \(\mathbb{Q}\) or \(\mathbb{Z}\) for which
\[
 ab+1,\qquad ac+1,\qquad bc+1,\qquad abc+1
\]
are all polynomial squares. We show that, up to permutation and a common invertible affine change of variable, there is exactly one such triple of quadratic polynomials in \(\mathbb{R}[X]\), the one arising from a recent construction of Jurasić, and that none exists in \(\mathbb{Q}[X]\) or
\(\mathbb{Z}[X]\). Moreover, 
for arbitrary nonconstant entries over $\R$ or $\Q$, we prove that the degree pattern  \((2n,2m,2l)\) with
\(1\le n\le m\le l\)  satisfies either
\(l=m\) or \(l>n+m\). For the case $l=m$, we derive
factorization systems and reducibility restrictions. In the equal-degree case over \(\Q\), we further show that
every entry has at least two distinct irreducible factors. As an application, we show that in a natural regular
parametrized family of polynomial Diophantine triples in \(\Q[X]\) the fourth expression $abc+1$
is never a square when the parameter is constant or the first entry is a square.
\end{abstract}



\section{Introduction and statement of results}

A set of \(m\) positive integers
\(\{a_1,\ldots,a_m\}\) is called a \emph{Diophantine \(m\)-tuple}
if \(a_i a_j+1\) is a
perfect square whenever \(i\neq j\). 
 Diophantine tuples have been studied extensively.
For an overview of their history, principal constructions, and connections
with elliptic curves, see \cite{DujellaBook2024}.

An important role in the theory of Diophantine tuples is played by regular
extensions. A Diophantine triple \(\{a,b,c\}\), with
\(ab+1=r^2\), is called \emph{regular} if, after a suitable ordering and
choice of the sign of \(r\),
\begin{equation}\label{eq:regular}
 c=a+b+2r.
\end{equation}
Similarly, a Diophantine quadruple \(\{a,b,c,d\}\), with
$ab+1=r^2,$ $ac+1=s^2$ and $bc+1=t^2,$
is called \emph{regular} if, after a suitable ordering and choices of
signs,
\[
 d=a+b+c+2abc+2rst.
\]
For more details on regular Diophantine triples and quadruples, see
\cite{DujellaBook2024,Dujella2001}.

Polynomial analogues of the Diophantine-tuple problem have also been
studied. In this setting the entries are polynomials, at least one of them
nonconstant, and one asks that the product of every two distinct entries,
increased by \(1\), be a polynomial square. Early and foundational work
includes \cite{Jones1976,DujellaFuchsTichy2002,DujellaFuchs2003,
DujellaFuchs2004}. Equal-degree and quadratic cases were developed further in
\cite{DujellaFuchsWalsh2006,DujellaJurasic2010,DujellaJurasic2011,
Jurasic2011}. 

The problem considered in the present paper imposes an additional square
condition. Namely, Dujella and Szalay \cite{DujellaSzalay} investigated the existence of three distinct positive integers \(a,b,c>1\)
such that
\[
        ab+1,\qquad ac+1,\qquad bc+1,\qquad abc+1
\]
are all perfect squares and proved that infinitely many
such integer triples exist. Their result
motivates the study of the corresponding problem over polynomial rings.

\begin{definition}
Let \(K\) be a field of characteristic zero. A set
\(\{a,b,c\}\subset K[X]\) of distinct nonzero polynomials is called a
\emph{four-square polynomial triple} if there are \(r,s,t,u\in K[X]\)
such that
\begin{equation}\label{eq:four-squares}
        ab+1=r^2,\qquad
        ac+1=s^2,\qquad
        bc+1=t^2,\qquad
        abc+1=u^2.
\end{equation}
\end{definition}

Jurasić \cite{Jurasic2026} recently constructed an example with three
nonconstant quadratic entries from \(\mathbb R[X]\) and showed that composition
with a nonconstant polynomial gives examples in every positive even common
degree. The quadratic parameters are algebraic but not rational. This raises
the natural questions of whether all real quadratic examples arise from the
same construction and whether nonconstant examples exist over
\(\mathbb Q[X]\) or \(\mathbb Z[X]\). While  \cite{Jurasic2026} gives existence, our first main result shows that, up to permutation and a common invertible real affine change of variable, there is exactly one four-square polynomial triple in $\mathbb{R}[X]$ whose three entries
have degree two. For a triple defined over \(\mathbb Q[X]\), the
affine normalization is defined over \(\mathbb Q\), and the resulting
parameter \(k\in\mathbb Q\) must satisfy an irreducible cubic over
\(\mathbb Q\). Consequently, the degree pattern \((2,2,2)\) is impossible
over \(\mathbb Q[X]\), and hence also over \(\mathbb Z[X]\). 
 
We also establish several general restrictions. For a four-square
polynomial triple with nonconstant entries in \(\Q[X]\), all three
degrees are positive and even, and each leading coefficient is a
positive rational square. Over either \(\R\) or \(\Q\), we prove
that any degree pattern \((2n,2m,2l)\), with
\(1\le n\le m\le l\), must satisfy either \(l=m\) or
\(l>n+m\).  For the degree pattern \((2n,2m,2m)\), we derive factorization
systems and reducibility restrictions. In the equal-degree case
\((2n,2n,2n)\) over \(\Q\), we further show that every entry has
at least two distinct irreducible factors. Finally, we apply
these results to a regular parametrized family of polynomial
Diophantine triples in \(\Q[X]\).
The computer-assisted algebraic computations were carried out using
Wolfram Mathematica, Version 14.3 \cite{Mathematica}.

The paper is organized as follows. Section~\ref{sec2} fixes the notation and
establishes the elementary restrictions on degrees and leading
coefficients. Section~\ref{sec3} establishes the uniqueness of the real quadratic four-square polynomial triple  up to permutation and affine change of variable,
identifies it with Jurasi\'c's triple, and proves the nonexistence
of  quadratic triples over \(\mathbb Q[X]\). Section~\ref{sec:degree-gap} develops the general degree
restrictions, the factor-allocation system for $(2n,2m,2m)$, and
its equal-degree specialization. It concludes with the application
to a regular parametrized family.

\section{Elementary restrictions}\label{sec2}

Throughout the paper, let
\(\{a,b,c\}\subset K[X]\) be a four-square polynomial triple whose entries are
nonconstant. This excludes constant four-square triples and the degenerate triples $\{a,b,c\}$ which contains $1$, analogously as it is
discussed in \cite{Jurasic2026}. Let
\(r,s,t,u\in K[X]\) be square roots chosen as in
\eqref{eq:four-squares}, and let \(\lc\) denote the leading   coefficient of a polynomial.
Polynomial greatest common divisors are taken to be monic.

The parity and positivity assertions over \(\R\) were observed
in \cite[Lemma~2.2]{Jurasic2026}. We record the following form,
which also identifies the leading coefficients as squares in
the coefficient field.

\begin{lemma}
\label{lem:degrees}
Let \(K\in\{\Q,\R\}\). If all three polynomials \(a,b,c\) in \eqref{eq:four-squares} are nonconstant, then their degrees are even and
their leading coefficients are nonzero squares in \(K\). In particular,
their leading coefficients are positive.

\end{lemma}

\begin{proof}
The conclusion about the parity of degrees of polynomials $a,b,c$ follows analogously as in \cite[Lemma~2.2]{Jurasic2026}. Moreover,
\[
 \lc(a)=\left(\frac{\lc(u)}{\lc(t)}\right)^2,
\]
and the analogous identities hold for \(b\) and \(c\). 
Thus, the leading coefficient of each of \(a,b,c\) is a nonzero square in \(K\). Since \(K\in\{\Q,\R\}\), they are positive.
\end{proof}

Further restrictions on the degrees will be developed in
Section~\ref{sec:degree-gap}. 

 \section{The quadratic case} \label{sec3}
 
In this section we treat four-square polynomial triples of quadratic polynomials. Jurasi\'c's construction \cite[Theorem~3.2]{Jurasic2026}
produces one such triple in \(\R[X]\), with entries determined
by the three real roots of a fixed cubic polynomial and by explicitly given
leading coefficients. We show that every quadratic four-square
polynomial triple in \(\R[X]\) can be transformed into that triple. 
To make
this precise, we call two triples \(\{a,b,c\},\{a',b',c'\}\subset K[X]\)
\emph{equivalent over \(K\in\{\Q,\R\}\)} if one is obtained from the other by a
permutation of the entries followed by a common invertible affine change of
variable \(X\mapsto\alpha X+\beta\), where \(\alpha\in K^*\) and
\(\beta\in K\).  Since permutations and invertible affine substitutions form a group, this is an equivalence
relation under which both the four-square
property \eqref{eq:four-squares} and the degrees of the entries
are preserved. The main result of
this section, \Cref{thm:quadratic-classification}, states that there is
exactly one equivalence class over \(\R\) of quadratic four-square polynomial triples
in \(\R[X]\), and \Cref{cor:no-rational-quadratic} shows the nonexistence of quadratic
four-square polynomial triples in \(\Q[X]\).

By \Cref{lem:degrees} and completion of the square, every quadratic
entry in a four-square polynomial triple can be written as the square of a
linear polynomial plus a constant. We begin by classifying a pair of such
quadratics.

\begin{lemma}
\label{lem:quadratic-pair} Let \(K\in\{\Q,\R\}\). Let \(\lambda,\mu\in K\), and let
\[
 f=P^2+\lambda,\qquad g=Q^2+\mu,
\]
where \(P,Q\in K[X]\) are linear polynomials. Write
$
 Q=mP+n,
$
with \(m\in K^*\) and \(n\in K\). Then \(fg+1\) is a square in \(K[X]\)
if and only if one of the following conditions holds:
\begin{enumerate}
 \item \(n=0\), and for some \(p\in K^*\),
 \[
  g=p^2f+2p,\qquad fg+1=(pf+1)^2;
 \]
 \item \(n\ne0\), and
 \[
  \lambda=-\frac1{n^2},\qquad
  \mu=-\frac{m^2}{n^2},\qquad
  fg+1=\left(PQ-\frac{m}{n^2}\right)^2.
 \]
\end{enumerate}
\end{lemma}

\begin{proof}
Let \(R\) be a square root of \(fg+1\). Then $R^2=P^2Q^2+\mu P^2+\lambda Q^2+\lambda\mu+1$,
so
\begin{equation}\label{RPQ}
 (R-PQ)(R+PQ)=\mu P^2+\lambda Q^2+\lambda\mu+1.
\end{equation}
The polynomial on the right-hand side of \eqref{RPQ} has degree at most two.
Choose the sign of \(R\) so that \(R\) and \(PQ\) have the same leading
coefficient. Then \(\deg(R-PQ)\leq1\). If \(\deg(R-PQ)=1\), then
\(\deg(R+PQ)=2\), so the left-hand side of \eqref{RPQ} has degree three, a
contradiction. Thus \(R=PQ+k\) for some \(k\in K\). Substituting
$R=PQ+k$ and $Q=mP+n$ into \eqref{RPQ} and comparing the coefficients of \(P^2\), \(P\), and \(1\), we obtain
\[
 \mu+\lambda m^2-2km=0,\qquad
 n(\lambda m-k)=0,\qquad
 \lambda n^2+\lambda\mu+1-k^2=0,
\]respectively.
If \(n=0\), these equations give
$k=\frac{\lambda m^2+\mu}{2m},$
$ (\mu-\lambda m^2)^2=4m^2.$ Set
$p=\frac{\mu-\lambda m^2}{2}.$
Then, \(p^2=m^2\). Since \(m\neq0\), this also gives \(p\neq0\). Moreover, \(\mu=p^2\lambda+2p\), and \(Q=mP\). Hence,
\begin{align*}
 g&=m^2P^2+\mu=p^2(P^2+\lambda)+2p=p^2f+2p,\\
 fg+1&=(pf+1)^2.
\end{align*}
If \(n\ne0\), the three coefficient equations give
\[
 k=\lambda m,\qquad \mu=\lambda m^2,\qquad \lambda n^2=-1,
\]
which is the second alternative and yields
\(R=PQ-m/n^2\). Conversely, direct substitution in either alternative gives
the displayed square identity.
\end{proof}

For a quadratic polynomial \(f=P^2+\lambda\in K[X]\), we define its center to be the unique zero of the linear polynomial \(P\).
Equivalently, the center is the $x$-coordinate of the vertex of
the graph of \(f\). In the
notation of the preceding lemma, \(Q=mP\) (equivalently, \(n=0\)) if and only if \(f\) and \(g\) have the same center.

\begin{lemma}
\label{lem:center-patterns}
Let \(K\in\{\mathbb{Q},\mathbb{R}\}\), and let
\(a,b,c\in K[X]\) be quadratic polynomials satisfying the
four-square conditions \eqref{eq:four-squares}. Then, the following center patterns are
impossible:

\begin{enumerate}
 \item exactly two of \(a,b,c\) have the same center;
 \item the three entries \(a,b,c\) have pairwise distinct centers.
\end{enumerate}

Consequently, every quadratic four-square polynomial triple over \(K[X]\)
have a common center.
\end{lemma}

\begin{proof}
By Lemma \ref{lem:degrees}, $\lc(a)$ is a square
in $K$. Hence, after a common invertible affine change of variable over $K$,
\(X\mapsto \alpha X+\beta\), where $\alpha\in K^*$ and $\beta\in K$, we may assume that
\[
 a=X^2+\lambda.
\]

Suppose first that exactly two entries have the same center. After
relabeling, assume that \(a\) and \(b\) have the same center, whereas
\(c\) has a different center. Applying \Cref{lem:quadratic-pair} to the pairs $(a,b)$ and $(a,c)$, using respectively the same-center and shifted-center alternatives, there are
\(p,\rho,n\in K^*\) such that
$b=p^2a+2p,$
 $\lambda=-\frac{1}{n^2},$
 $c=(\rho X+n)^2-\frac{\rho^2}{n^2}.$
Thus $a=X^2-\frac{1}{n^2}$ and
 $b=p^2X^2-\frac{p^2}{n^2}+2p.$ Set \(P=pX\). Then,
\[
 b=P^2-\frac{p^2}{n^2}+2p,
 \qquad
 c=\left(\frac{\rho}{p}P+n\right)^2
      -\frac{\rho^2}{n^2}.
\]
The polynomials \(b\) and \(c\) have different centers, so the
second, shifted-center alternative of \Cref{lem:quadratic-pair} applied to this pair yields
\[
  -\frac{p^2}{n^2}+2p=-\frac{1}{n^2},\quad  -\frac{\rho^2}{n^2}
 =
 -\frac{(\rho/p)^2}{n^2}.
\]
Since \(\rho n\neq0\), the second equality implies \(p^2=1\).
Substituting this into the first equality gives \(2p=0\), which contradicts $p\in K^*$.
Hence, the pattern with exactly two common centers cannot occur.

\medskip

It remains to consider the case of three pairwise distinct centers. Apply the
shifted-center alternative of \Cref{lem:quadratic-pair} first to the pairs
\((a,b)\) and \((a,c)\). We
first allow different shift parameters \(n_b\) and $n_c$:
$a=X^2+\lambda,$
 $b=(mX+n_b)^2-\frac{m^2}{n_b^2},$
$ c=(\ell X+n_c)^2-\frac{\ell^2}{n_c^2}$, 
where \(m,\ell,n_b,n_c\in K^*\), and
$\lambda=-\frac1{n_b^2}=-\frac1{n_c^2}.$
Thus \(n_c^2=n_b^2\). Replacing the linear polynomial \(\ell X+n_c\) by its negative leaves \(c\) unchanged, so, after replacing \(\ell\) by \(-\ell\) if necessary, we may assume \(n_c=n_b=:n\). Hence,
\[
 a=X^2-\frac1{n^2},\qquad
 b=(mX+n)^2-\frac{m^2}{n^2},\qquad
 c=(\ell X+n)^2-\frac{\ell^2}{n^2}.
\]
To apply \Cref{lem:quadratic-pair} to the pair \((b,c)\), put \(P=mX+n\). Then,
$\ell X+n=\frac{\ell}{m}P+n-\frac{\ell n}{m}$
and
\begin{align*}
 b =P^2-\frac{m^2}{n^2}, \qquad 
 c =\left(\frac{\ell}{m}P+n-\frac{\ell n}{m}\right)^2
      -\frac{\ell^2}{n^2}.
\end{align*}
Since \(b\) and \(c\) have distinct centers, \(n-\frac{\ell n}{m}\neq0\). The shifted-center alternative gives
$ -\frac{m^2}{n^2}
 =-\frac1{\left(n-\frac{\ell n}{m}\right)^2},$
and therefore
\begin{equation}\label{nc}
 \left(n-\frac{\ell n}{m}\right)^2=\frac{n^2}{m^2}.
\end{equation}
Since \(m,n\neq0\), equation \eqref{nc} is equivalent to
\((m-\ell)^2=1\). Thus the three-distinct-center case has the normal form
\[
 \begin{aligned}
 a&=X^2-\frac{1}{n^2},\\
 b&=(mX+n)^2-\frac{m^2}{n^2},\\
 c&=(\ell X+n)^2-\frac{\ell^2}{n^2},
 \end{aligned}
 \qquad (m-\ell)^2=1.
\]
The assumptions \(m,\ell,n\ne0\) and \((m-\ell)^2=1\) ensure that the
three centers are pairwise distinct. For these parameters,
\(\{a,b,c\}\) is a polynomial Diophantine triple, since
\begin{align*}
 ab+1&=\left(X(mX+n)-\frac{m}{n^2}\right)^2,\\
 ac+1&=\left(X(\ell X+n)-\frac{\ell}{n^2}\right)^2,\\
 bc+1&=\left((mX+n)(\ell X+n)-\frac{m\ell}{n^2}\right)^2.
\end{align*}

Put \(S=m+\ell\) and \(w=m\ell\). Then, \(S^2=4w+1\). Also,
\begin{align*}
 M&=X(mX+n)(\ell X+n)\\
  &=wX^3+nSX^2+n^2X.
\end{align*}
Moreover, $\deg\bigl(abc+1-M^2\bigr)\leq 4.$
If \(abc+1=v^2\), for some $v\in K[X]$, then
\[
 abc+1-M^2=(v-M)(v+M).
\]
Both \(v\) and \(M\) have degree \(3\). Choose the sign of \(v\) so
that \(v\) and \(M\) have the same leading coefficient. Then, \(\deg(v+M)=3\). Since $ \deg\bigl((v-M)(v+M)\bigr) =\deg\bigl(abc+1-M^2\bigr)\le4, $
either \(v=M\) or \(\deg(v-M)\le1\). Hence, there exist \(\rho,\tau\in K\) such that
$ v=M+\rho X+\tau, $ that is, $$v=wX^3+nSX^2+(n^2+\rho)X+\tau.$$ Comparison of the coefficients of \(X^4\) and \(X^3\) in \(abc+1=v^2\) gives
\[
 \rho=-\frac{3w}{2n^2},
 \qquad
 \tau=-\frac{S}{2n},
\]respectively. Using \(S^2=4w+1\), the remaining coefficient equations, corresponding to the coefficients of \(X^2\), \(X\), and \(1\), respectively, reduce to
\begin{align}
 3w^2&=4n^4(w+1), \label{eq:three-centers-1}\\
 S(w-2n^4)&=0, \label{eq:three-centers-2}\\
 4n^8-4n^6-4n^4w-3n^4+4w^2&=0.
 \label{eq:three-centers-3}
\end{align}
If \(S=0\), then \(S^2=4w+1\) gives \(w=-1/4\). Equation
\eqref{eq:three-centers-1} then gives \(n^4=1/16\). Hence, \(n^2=1/4\). Substituting into \(\eqref{eq:three-centers-3}\), its left-hand side equals \(5/64\), a contradiction. If \(S\neq0\), then \eqref{eq:three-centers-2} gives
\(w=2n^4\). Substitution in \eqref{eq:three-centers-1} yields
\(n^4=1\), so \(w=2\) and $n^2=1$. The left-hand side of
\eqref{eq:three-centers-3} then equals \(5\), again a contradiction.
Therefore, the case of three pairwise distinct centers is also impossible.

The only remaining possibility is that all three quadratic entries have a
common center.
\end{proof}

The following lemma gives a complete parametrization of the case
in which all three quadratic entries have the same center. It leads
to the cubic equation \eqref{eq:common-G} for a parameter \(k\).
As we shall see in \Cref{prop:identification-Jurasic} and \Cref{rem:three-roots},
its three real roots yield equivalent triples, corresponding to the
three possible choices of the entry normalized to be monic.

\begin{lemma}
\label{lem:common-center-family}
Let \(K\in\{\Q,\R\}\), and let \(a,b,c\in K[X]\) be  
quadratic polynomials with a common center. Then \(a,b,c\) satisfy the four-square conditions \eqref{eq:four-squares} if and only if, after a common invertible affine change of variable and, if necessary, a permutation of \(b\) and \(c\), there exists \(k\in K\) satisfying
\begin{equation}\label{eq:common-G}
 G(k):=16k^3+55k^2-64k+16=0,
\end{equation}
such that, with
\begin{equation}\label{eq:common-parameters}
 \sigma=\frac{k(8k-3)}{4(2k-1)},
\end{equation}
and
\begin{equation}\label{eq:common-pq}
 p=\frac{\sigma+1}{2},\qquad
 q=\frac{\sigma-1}{2},
\end{equation}
the triple has the form
\begin{equation}\label{eq:common-triple}
 a=X^2-\frac1{k^2},\qquad
 b=p^2a+2p,
 \qquad
 c=q^2a+2q.
\end{equation}
In this case, the four square roots may be chosen as
\begin{equation}\label{eq:common-four-roots}
 r=pa+1,\quad  s=qa+1,\quad t=pqa+p+q,\quad  u=X(pqa+k).
\end{equation}
Moreover, the root \(k\) of \(G\) is uniquely determined by the normalized
triple \eqref{eq:common-triple}.
\end{lemma}

\begin{proof}
Assume that \(a,b,c\) satisfy the four-square conditions
\eqref{eq:four-squares}. By \Cref{lem:degrees}, the leading
coefficient of \(a\) is a nonzero square in \(K\). Therefore, after
completing the square by a common  invertible affine change of variable over
\(K\), we may assume that
\[
 a=X^2+\lambda,
 \qquad \lambda\in K.
\]
Since all three centers coincide, the first alternative of
\Cref{lem:quadratic-pair},
applied to the pairs \((a,b)\) and \((a,c)\), gives
\begin{equation}\label{eq:common-bc-initial}
 b=p^2a+2p,\qquad c=q^2a+2q
\end{equation}
for some \(p,q\in K^*\).  In particular,
\[
 ab+1=(pa+1)^2,
 \qquad
 ac+1=(qa+1)^2.
\]
By \eqref{eq:common-bc-initial}, we obtain
\begin{align}
 bc+1
  &=(pqa+p+q)^2-(p-q)^2+1.\label{eq:common-bc-square}
\end{align}
Hence, \(bc+1\) is a square if \((p-q)^2=1\).
Conversely, suppose \(bc+1=N^2\) for some $N\in K[X]$, and put
\(M=pqa+p+q\) and \(\kappa=1-(p-q)^2\in K\). Equation
\eqref{eq:common-bc-square} gives
$(N-M)(N+M)=N^2-M^2=\kappa.$
If \(\kappa\ne0\), then both \(N-M\) and \(N+M\) are units in \(K[X]\),
hence constants. It follows that \(M\) is constant, which is impossible because \(a\) is nonconstant and
\(pq\ne0\).  Therefore \(\kappa=0\), and
\begin{equation}\label{eq:p-minus-q}
 (p-q)^2=1.
\end{equation}
After interchanging \(b\) and \(c\), equivalently \(p\) and \(q\), if necessary, we may assume
\(p-q=1\).

Set \(\pi=pq\) and \(\sigma=p+q\). Then
\eqref{eq:p-minus-q} is equivalent to
\begin{equation}\label{eq:sigma-pi-common}
 \sigma^2-4\pi=1.
\end{equation}
Furthermore, by \eqref{eq:common-bc-initial},
\begin{align}
 abc+1
  &=\pi^2a^3+2\pi\sigma a^2+4\pi a+1.
  \label{eq:abc-common-expanded}
\end{align}
The right-hand side of \eqref{eq:abc-common-expanded} is an even polynomial in
\(X\). Hence, if \(abc+1=u^2\), then \([u(-X)]^2=[u(X)]^2\), so
\(u(-X)=u(X)\) or \(u(-X)=-u(X)\).  Since $u$ has odd degree $3$, it cannot be an even
polynomial. Hence $u(-X)=-u(X)$, that is, $u=\alpha X^3+\beta X=X(\alpha X^2+\beta)$
for some \(\alpha,\beta\in K\). The leading term of \(abc+1\) is \(\pi^2X^6\), so the leading coefficient of \(u\) is \(\pm\pi\).
Replacing \(u\) by \(-u\), if necessary, we may assume that its
leading coefficient is \(\pi\) and
\(u=X(\pi X^2+\beta)\). Since \(a=X^2+\lambda\), we have
\(\pi X^2+\beta=\pi a+k\), where
\(k=\beta-\pi\lambda\in K\). Thus
\begin{equation}\label{eq:u-common}
 u=X(\pi a+k)
\end{equation}
for some \(k\in K\). Since \(X^2=a-\lambda\), squaring \eqref{eq:u-common} gives
\begin{align}
 u^2
  &=\pi^2a^3+(2\pi k-\lambda\pi^2)a^2
    +(k^2-2\lambda\pi k)a-\lambda k^2.
 \label{eq:u-common-expanded}
\end{align}
Since \(a\) is nonconstant, we may compare coefficients as polynomials in \(a\). Comparing the coefficients of \(a^2,a,1\) in
\eqref{eq:abc-common-expanded} and \eqref{eq:u-common-expanded} yields
\begin{align}
 2k-\lambda\pi&=2\sigma,
 \label{eq:common-coeff-1}\\
 k^2-2\lambda\pi k&=4\pi,
 \label{eq:common-coeff-2}\\
 -\lambda k^2&=1.
 \label{eq:common-coeff-3}
\end{align}
Equations \eqref{eq:common-coeff-3}, \eqref{eq:common-coeff-2},
and \eqref{eq:common-coeff-1} first imply that
\(k\notin\{0,\frac12\}\) and then yield, respectively,
$\lambda=-\frac1{k^2},$ $\pi=\frac{k^3}{2(2k-1)},$ and  $\sigma=\frac{k(8k-3)}{4(2k-1)}.$ Note that \(\lambda\), \(\pi\), and \(\sigma\) are determined by the
normalized triple, so \eqref{eq:common-coeff-1} shows that \(k\) is uniquely
determined by it, as claimed. 
Thus, $\sigma$ is given by \eqref{eq:common-parameters}. Finally, substituting these expressions into
\eqref{eq:sigma-pi-common} gives
\begin{equation*}
 16k^3+55k^2-64k+16=0,
\end{equation*}
which is \eqref{eq:common-G}. Since a root of \(G\) cannot satisfy
\(k=0\) or \(2k-1=0\), all displayed expressions are defined.

Conversely, let \(k\in K\) satisfy \eqref{eq:common-G}. Define
\(\sigma\) by \eqref{eq:common-parameters} and 
\(p,q\) by \eqref{eq:common-pq}, and set $\pi=\frac{k^3}{2(2k-1)}$.  Since
$ \sigma^2-4\pi-1 =-\frac{G(k)}{16(2k-1)^2}, $ we have \(\sigma^2-4\pi=1\). Hence,
$$ p-q=1,\qquad p+q=\sigma,\qquad pq=\pi. $$ Since \(k\notin\{0,\frac12\}\), we have \(\pi\neq0\), so \(p,q\neq0\). The entries are pairwise distinct: for example, \(a=b\) would
force \((p^2-1)a+2p=0\), which is impossible for nonconstant \(a\) and nonzero
\(p\). Similarly, \(a\neq c\). Moreover, \(b-c=(p-q)((p+q)a+2)\ne0\). Direct substitution then gives all four
identities in \eqref{eq:common-four-roots}. This proves both necessity and
sufficiency.
\end{proof}

\begin{lemma}
\label{lem:G-roots}
The polynomial \eqref{eq:common-G} is irreducible over \(\Q\) and
has three distinct real roots  $k_1<k_2<k_3$,
\[
k_1\approx-4.39857,\quad k_2\approx0.42080,\quad k_3\approx0.54027.
\]
\end{lemma}
\begin{proof}
One has
\[
 G(k)\equiv k^3+k^2-k+1 \pmod 3.
\]
For \(k=0,1,2\), this is congruent to \(1,2,2\) modulo $3$, respectively, so $G$ has no root in
\(\mathbb F_3\). Being cubic, it is irreducible over \(\mathbb F_3\). Since
\(G\) is primitive and its reduction modulo \(3\) has the same degree and is
irreducible over \(\mathbb F_3\), Gauss's lemma \cite{Lang} implies that
\(G\) is irreducible over \(\Q\). Furthermore,
\[
 \disc(G)=529984=728^2>0,
\]
so \(G\) has three distinct real roots. The approximate values were computed with Mathematica.
\end{proof}

We now show that, for each real root \(k\) of \(G\), the triple
\eqref{eq:common-triple} coincides with the triple constructed by
Jurasi\'c \cite{Jurasic2026}, up to a rescaling of the variable
and a permutation of the entries.

\begin{proposition}
\label{prop:identification-Jurasic}
Let \(k\in\R\) be a root of the polynomial \(G\) in \eqref{eq:common-G}, and let \(\{a,b,c\}\) be the real four-square polynomial triple
\eqref{eq:common-triple}. Then there is a real number \(h>0\) such that,
after the affine change of variable \(X=\sqrt{h}\,z\) and a permutation of
the entries, the triple becomes
\begin{equation}\label{eq:Jurasic-triple}
 a=A(z^2-\alpha),\qquad
 b=B(z^2-\beta),\qquad
 c=C(z^2-\gamma),
\end{equation}
where \(\alpha<\beta<\gamma\) are the three real roots of the polynomial
\begin{equation}\label{eq:Jurasic-polynomial}
 f(Y)=Y^3-2Y^2+Y-\frac4{91},
\end{equation}
and
\begin{equation}\label{eq:Jurasic-ABC}
 A=\frac{2(\gamma-\beta)}{(\beta-\alpha)(\gamma-\alpha)},\qquad
 B=\frac{2(\gamma-\alpha)}{(\beta-\alpha)(\gamma-\beta)},\qquad
 C=\frac{2(\beta-\alpha)}{(\gamma-\alpha)(\gamma-\beta)}.
\end{equation}
The triple \eqref{eq:Jurasic-triple} is the one constructed in
\cite[Theorem~3.2]{Jurasic2026}; approximate values of
\(\alpha,\beta,\gamma\) and \(A,B,C\) are listed in
\cite[Remark~3.3]{Jurasic2026}. In particular, the triples
\eqref{eq:common-triple} corresponding to the three real roots of \(G\) are
equivalent over \(\R\) to each other.
\end{proposition}

\begin{proof}
Let \(k\in\mathbb{R}\) be a root of the polynomial \(G\) given by \eqref{eq:common-G}, and set
$\pi=pq=\frac{k^3}{2(2k-1)}.$
Put
\begin{equation}\label{h}
 h=\frac{3-4k}{k^2}
   =\frac1{k^2}-\frac{k}{\pi}.
\end{equation}
A direct calculation gives
$\pi^2h^3-\frac{91}{4}
   =-\frac{G(k)}{(2k-1)^2}.$
Since \(G(k)=0\), it follows that
$ \pi^2h^3=\frac{91}{4}.$
In particular, \(h>0\). We may therefore make the real affine change of variable $ X=\sqrt{h}\,z$, and we put $Y=z^2$.

Since
$ a=X^2-\frac1{k^2},$
$ b=p^2a+2p,$ 
$ c=q^2a+2q,$
and \(X^2=hY\), each of \(a,b,c\) is a linear polynomial in \(Y\).
Write
\[
 a=A(Y-\alpha),\qquad
 b=B(Y-\beta),\qquad
 c=C(Y-\gamma),
\]
where \(\alpha,\beta,\gamma\in\R\) and
\[
 A=h,\qquad B=p^2h,\qquad C=q^2h.
\]
Since \(h>0\), we have \(A,B,C>0\).

Using
\[
 u=X(\pi a+k),\qquad a=X^2-\frac1{k^2},
\]
together with \(X^2=hY\) and \eqref{h}, we obtain
\begin{align}
 abc+1
 &=\bigl[X(\pi a+k)\bigr]^2 \nonumber\\
 &=\pi^2h^3Y(Y-1)^2 \nonumber\\
 &=\frac{91}{4}Y(Y-1)^2.
 \label{y914}
\end{align}
Hence,
\[
 abc
 =\frac{91}{4}
 \left(Y^3-2Y^2+Y-\frac4{91}\right).
\]
On the other hand,
$ABC=(pq)^2h^3=\pi^2h^3=\frac{91}{4},$
so
\begin{equation}\label{abc1}
 abc=\frac{91}{4}(Y-\alpha)(Y-\beta)(Y-\gamma).
\end{equation}
Equations \eqref{y914} and \eqref{abc1} yield
$(Y-\alpha)(Y-\beta)(Y-\gamma)
 =Y^3-2Y^2+Y-\frac4{91}.$
Thus \(\alpha,\beta,\gamma\) are the roots of the polynomial \(f\) in
\eqref{eq:Jurasic-polynomial}.
Since
$\disc(f)=\frac{1024}{8281}
         =\left(\frac{32}{91}\right)^2>0,$
the roots \(\alpha,\beta,\gamma\) are pairwise distinct. After
simultaneously permuting the pairs
$ (A,\alpha),$ $ (B,\beta)$, $(C,\gamma),$
if necessary, we may assume that
\[
 \alpha<\beta<\gamma.
\]

By \cite[Lemma~3.1]{Jurasic2026}, if $\xi \neq \eta$ and $D > 0$, then $1 + D(Y - \xi)(Y - \eta)$ is the square of a linear polynomial in $Y$ if and only if \begin{equation}\label{eq:elementary-D-condition}
 D=\frac4{(\eta - \xi)^2}.
\end{equation}
Since $ ab+1=AB(Y-\alpha)(Y-\beta)+1$
is a square, \eqref{eq:elementary-D-condition} gives
\begin{equation}\label{eq:AB-Jurasic}
 AB=\frac4{(\beta-\alpha)^2}.
\end{equation}
Similarly, the square conditions for \(ac+1\) and \(bc+1\) give
\begin{equation}\label{eq:AC-BC-Jurasic}
 AC=\frac4{(\gamma-\alpha)^2},
 \qquad
 BC=\frac4{(\gamma-\beta)^2}.
\end{equation}
From \eqref{eq:AB-Jurasic} and \eqref{eq:AC-BC-Jurasic}, we obtain
$ A^2=\frac{(AB)(AC)}{BC}
 =\frac{4(\gamma-\beta)^2}
 {(\beta-\alpha)^2(\gamma-\alpha)^2}.$
Since \(A>0\),
\begin{equation}\label{eq:A-Jurasic}
 A=
 \frac{2(\gamma-\beta)}
 {(\beta-\alpha)(\gamma-\alpha)}.
\end{equation}
Likewise,
\begin{equation}\label{eq:BC-Jurasic}
 B=
 \frac{2(\gamma-\alpha)}
 {(\beta-\alpha)(\gamma-\beta)},
 \qquad
 C=
 \frac{2(\beta-\alpha)}
 {(\gamma-\alpha)(\gamma-\beta)}.
\end{equation}
Therefore, the identities \eqref{eq:AB-Jurasic} and
\eqref{eq:AC-BC-Jurasic} determine \(A,B,C\) uniquely in terms of
\(\alpha,\beta,\gamma\), these are the formulas \eqref{eq:Jurasic-ABC}, and in the variable \(z\)
the triple \(\{a,b,c\}\) takes the form \eqref{eq:Jurasic-triple}.

The cubic \eqref{eq:Jurasic-polynomial}, its ordered roots
\(\alpha<\beta<\gamma\), and the coefficients \eqref{eq:A-Jurasic} and
\eqref{eq:BC-Jurasic} are exactly the polynomial and the leading coefficients
appearing in Jurasi\'c's construction \cite[Theorem~3.2]{Jurasic2026}, so
\eqref{eq:Jurasic-triple} is the triple constructed there.
Finally, none of \(\alpha,\beta,\gamma,A,B,C\) depends on the chosen root
\(k\) of \(G\), only the scaling factor \(h\) in \eqref{h} does. Hence, the
triples \eqref{eq:common-triple} obtained from the three real roots of \(G\)
are all equivalent to \eqref{eq:Jurasic-triple}, and therefore to each other.

\end{proof}

We can now state the main result of this section.

\begin{theorem}
\label{thm:quadratic-classification}
The four-square polynomial triples in \(\R[X]\) whose entries
all have degree two form a single equivalence class over \(\R\),
represented by the triple \eqref{eq:Jurasic-triple} constructed
by Jurasi\'c \cite[Theorem~3.2]{Jurasic2026}.
\end{theorem}

\begin{proof}
Existence follows from \cite[Theorem~3.2]{Jurasic2026}. It also follows from
our results: since \(G\) has a real root \(k\) by \Cref{lem:G-roots}, the
sufficiency part of \Cref{lem:common-center-family} shows that
\eqref{eq:common-triple} is then a four-square polynomial triple in
\(\R[X]\), and \Cref{prop:identification-Jurasic} identifies it with
\eqref{eq:Jurasic-triple}.

Let \(\{a,b,c\}\subset\R[X]\) be a four-square polynomial
triple with three quadratic entries. By \Cref{lem:center-patterns}, the
cases in which exactly two entries have a common center or the three entries
have pairwise distinct centers are impossible, so the three entries have a
common center. By \Cref{lem:common-center-family}, the triple \(\{a,b,c\}\) is
equivalent over \(\R\) to the triple \eqref{eq:common-triple} for some real
root \(k\) of \(G\), and by \Cref{prop:identification-Jurasic} the latter is
equivalent over \(\R\) to \eqref{eq:Jurasic-triple}. Since equivalence over
\(\R\) is an equivalence relation, the triple \(\{a,b,c\}\) is equivalent to
\eqref{eq:Jurasic-triple}, which proves the theorem.
\end{proof}

\begin{remark}
\label{rem:three-roots}
Although the cubic \(G\) in \eqref{eq:common-G} has three
distinct real roots, the three corresponding triples \eqref{eq:common-triple} are equivalent over $\mathbb{R}$ by \Cref{thm:quadratic-classification}. The reason is that the normal form
\eqref{eq:common-triple} distinguishes one entry, \(a\), which is normalized
to be monic with center \(0\), while the order of the remaining two entries
is fixed by \(p-q=1\). By \Cref{lem:common-center-family}, the root \(k\) is
uniquely determined once this choice is made, so the three roots of \(G\)
correspond to the three possible choices of the distinguished entry. This can
also be seen directly. The numbers \(1/k_1^2,1/k_2^2,1/k_3^2\) are the roots of the polynomial 
\[
 E(T)=256T^3-2336T^2+5073T-256,
\]
and a direct computation shows that, for \(p,q\) given by
\eqref{eq:common-parameters} and \eqref{eq:common-pq}, we have 
\[
 E\Big(\frac{p^2}{k^2}-2p\Big)\equiv0 \pmod{G(k)},\qquad
 E\Big(\frac{q^2}{k^2}-2q\Big)\equiv0
 \pmod{G(k)}.
\]
Thus the constant terms \(-p^2/k^2+2p\) and \(-q^2/k^2+2q\) of \(b\) and
\(c\) in \eqref{eq:common-triple}, respectively, are \(-1/k'^2\) and \(-1/k''^2\), where
\(k',k''\) are the two remaining roots of \(G\). Since constant terms are
unchanged by the scaling \(X\mapsto X/p\) that makes \(b\) monic, the entry
\(b\) plays the role of \(a\) for the root \(k'\), and similarly for \(c\).
Consequently, in the normalization of \Cref{lem:common-center-family}, the
three entries of the unique triple have constant terms
\(-1/k_1^2\), \(-1/k_2^2\), and \(-1/k_3^2\), a description in which no entry
is distinguished.
\end{remark}
\medskip

The parametrization also yields the nonexistence of quadratic
four-square polynomial triples over \(\Q\).

\begin{corollary}
\label{cor:no-rational-quadratic}
There is no four-square polynomial triple with three quadratic
entries in \(\Q[X]\), and hence none in \(\Z[X]\).
\end{corollary}
\begin{proof}
Suppose that such a triple exists in \(\Q[X]\). By
\Cref{lem:center-patterns}, its three entries have a common center.
Hence, by \Cref{lem:common-center-family} applied with \(K=\Q\),
there exists \(k\in\Q\) satisfying \eqref{eq:common-G}.
This contradicts the irreducibility of \(G\) over \(\Q\)
established in \Cref{lem:G-roots}. The assertion over \(\Z[X]\)
follows from \(\Z[X]\subset\Q[X]\).
\end{proof}

\section{Degree restrictions and factorization systems}
\label{sec:degree-gap}

In this section, we study further restrictions on the degree patterns
and structure of four-square polynomial triples. We first derive
a degree gap from the regular-extension identities and the fourth
square condition. For the degree pattern \((2n,2m,2m)\) with
\(1\le n\le m\) , we then
obtain a factor-allocation system, which we refine in the equal-degree
case \((2n,2n,2n)\) using a composition obstruction. We conclude
with an application to a regular parametrized family.

\subsection{Regular extensions and degree restrictions}
\label{subsec:regularity-degree-gap}

Regular extensions turn the square identities into degree restrictions.
We first obtain a criterion for regularity, then exclude its boundary
case using the fourth square condition.

For a polynomial Diophantine triple \(\{a,b,c\}\), define
\[
 d_{\pm}=a+b+c+2abc\pm2rst.
\]
Direct expansion gives
\begin{equation*}
 ad_{\pm}+1=(at\pm rs)^2,\ \
 bd_{\pm}+1=(bs\pm rt)^2,\ \
 cd_{\pm}+1=(cr\pm st)^2,
\end{equation*}
and
\begin{equation}\label{eq:dproduct}
 d_+d_-=(c-a-b-2r)(c-a-b+2r).
\end{equation}

\begin{lemma}
\label{lem:regularity}
Let \(K\in\{\Q,\R\}\), and let \(a,b,c\in K[X]\) satisfy
\eqref{eq:four-squares} with
$
0<\deg a\le \deg b\le \deg c.
$
If \(\deg a+\deg b>\deg c\), then, after changing the sign of \(r\)
if necessary, \eqref{eq:regular} holds.
\end{lemma}

\begin{proof}
Assume \(\deg a+\deg b>\deg c\).
Choose the signs of \(r,s,t\) so that their leading coefficients
are positive. By \Cref{lem:degrees}, the leading coefficients of
\(a,b,c\) are positive as well. The leading terms of
\(2abc\) and \(2rst\) in \(d_+\) add, so
\[
 \deg d_+=\deg a+\deg b+\deg c.
\]
Since \(\deg r=\frac{\deg a+\deg b}{2}\le \deg c\), the right side of
\eqref{eq:dproduct} has degree at most \(2\deg c\).
Since \(\deg a+\deg b>\deg c\), a nonzero \(d_-\) would
make the left side have degree \(>2\deg c\).
Thus \(d_-=0\), and \eqref{eq:dproduct} gives
\(c=a+b+2r\) or \(c=a+b-2r\). In the latter case replace \(r\)
by \(-r\), which preserves \(r^2=ab+1\) and gives
\eqref{eq:regular}. 
\end{proof}

\begin{remark}
 \Cref{lem:regularity} shows that every four-square polynomial
triple over \(K\in\{\Q,\R\}\) with
\(0<\deg a\leq\deg b\leq\deg c\) satisfying
\(\deg a+\deg b>\deg c\) is regular.
In particular, this applies to every triple with degree pattern
\((2n,2m,2m)\), \(1\le n\le m\). Hence,
the equal-degree pattern \((2n,2n,2n)\) is included. 
\end{remark}

The preceding lemma treats the strict inequality
$\deg a+\deg b>\deg c$. We now exclude the boundary case
$\deg c=\deg a+\deg b$. 

\begin{proposition}\label{prop:boundary-degree-obstruction}
Let \(K\in\{\Q,\R\}\), and let \(a,b,c\in K[X]\)
be a four-square polynomial triple with nonconstant entries.
Then
\[
 \deg c\ne\deg a+\deg b.
\]
In particular, the degree patterns \((2n,2m,2(n+m))\) are impossible
for all positive integers \(n,m\).
\end{proposition}

\begin{proof}
It suffices to work over $\R$. Suppose that
$\deg a\le\deg b$ and $\deg c=\deg a+\deg b$.
By \Cref{lem:degrees}, all three leading coefficients are positive.
Choose the signs of $r,s,t,u$ in \eqref{eq:four-squares}
so that all four leading coefficients are positive. 
Put
\[
e=d_-=a+b+c+2abc-2rst.
\]
We first determine its degree directly. Identity
\eqref{eq:dproduct} gives
\[
d_+e=(c-a-b)^2-4(ab+1).
\]
Since $\deg c>\deg a$ and $\deg c>\deg b$, the right-hand side has degree $2\deg c$ and leading
coefficient $\lc(c)^2$, whereas $d_+$ has degree $2\deg c$ and leading
coefficient $4\lc(a)\lc(b)\lc(c)$.
The right-hand side is nonzero, so $e\ne0$.
Thus $e$ is constant, and comparison of leading coefficients gives
\begin{equation}\label{eq:boundary-positive-e}
e=\frac{\lc(c)}{4\lc(a)\lc(b)}>0.
\end{equation}

Put $v=cr-st$. The extension identities preceding
\Cref{lem:regularity}, together with $r^2=ab+1$, give
\begin{equation}\label{eq:boundary-rv}
v^2=ce+1,\qquad 2rv=2c(ab+1)-2rst=c-a-b+e.
\end{equation}
Thus $\deg r=\deg v=\frac{\deg c}{2}$.
Since
$2\lc(r)\lc(v)=\lc(c)>0$ and $\lc(r)>0$, it follows that $\lc(v)>0$.

We have $r^2v^2=(ab+1)(ce+1)=eabc+r^2+v^2-1.$ Together with \(u^2=abc+1\), this gives
\begin{equation}\label{eq:boundary-difference}
 (\sqrt e\,u-rv)(\sqrt e\,u+rv)
 =eu^2-r^2v^2=e+1-r^2-v^2.
\end{equation}
The right-hand side has degree $\deg c$. The second factor on the left
also has degree $\deg c$. 
 Therefore $\sqrt e\,u-rv=\kappa$ is a nonzero real
constant, and
\begin{equation}\label{eq:boundary-conic}
r^2+2\kappa rv+v^2=e+1-\kappa^2.
\end{equation}

If $\kappa^2\ne1$, the left-hand side factors over $\C$ as
$(r+\alpha v)(r+\beta v)$, with $\alpha\ne\beta$.
If the right-hand side is nonzero, both factors are constant polynomials in \(\C[X]\) forcing $(\alpha-\beta)v\in\C$, that is,
 $v$ is a constant. If the right-hand side is zero,
one factor vanishes, so  $r=\lambda v$, where $\lambda$ is a real constant since $r$ and $v$ are
real polynomials.
In this case \eqref{eq:boundary-rv} yields
\[
a+b=\left(\frac1e-2\lambda\right)v^2+e-\frac1e.
\]
Since $\deg(a+b)=\deg b<\deg c=\deg v^2$, the coefficient of $v^2$ must
vanish, making $a+b$ constant. Both alternatives are impossible. Hence $\kappa^2=1$. 

The case $\kappa=1$ would give
$(r+v)^2=e$, contrary to the positive leading coefficients of $r,v$.
Thus $\kappa=-1$ and $r=v+\delta,$ $\delta\in\R$, $\delta^2=e.$
Equation \eqref{eq:boundary-rv} now gives
$a+b=\left(\frac1e-2\right)v^2-2\delta v+e-\frac1e.$
The same degree comparison as earlier forces $e=1/2$. Consequently,
\[
a+b=-2\delta v-\frac32,\qquad
ab=(v+\delta)^2-1,\qquad \delta^2=\frac12.
\]
It follows that
\[
(a-b)^2=-2v^2-2\delta v+\frac{17}{4},
\]
where right-hand side has negative leading coefficient. This cannot
be a square in $\R[X]$, completing the proof.
\end{proof}

\begin{corollary}\label{cor:general-degree-gap}
For a four-square polynomial triple over \(K\in\{\Q,\R\}\)
with \(0<\deg a\le\deg b\le\deg c\), one has
\[
 \deg c=\deg b
 \qquad\text{or}\qquad
 \deg c>\deg a+\deg b.
\]
Equivalently, a degree pattern \((2n,2m,2l)\) with
\(1\le n\le m\le l\) must satisfy
\(l=m\) or \(l>n+m\).
\end{corollary}

\begin{proof} 
The equality $\deg c=\deg a+\deg b$ is excluded by
\Cref{prop:boundary-degree-obstruction}. If $\deg c<\deg a+\deg b$,
\Cref{lem:regularity} allows us to choose the sign of the root
$r^2=ab+1$ so that $c=a+b+2r$.
Since $\deg r=\frac{\deg a+\deg b}{2}\le \deg b$, it follows that $\deg c\le \deg b$,
and hence $\deg c=\deg b$.
\end{proof}
\medskip
In particular, \((2,2,6)\) is the smallest
degree pattern with  \(\deg c>\deg a+\deg b\) that is not excluded by Corollary \ref{cor:general-degree-gap}. Whether it
can be realized remains open.

\subsection{Factor allocation with two entries of maximum degree}
\label{subsec:two-largest}

The first alternative in \Cref{cor:general-degree-gap} leads to
a factor-allocation system that will be specialized to equal degrees
in Subsection~\ref{sec:equal-degrees}. 

\begin{proposition}\label{prop:two-largest-split}
Let $K\in\{\Q,\R\}$, and let $n,m$ be integers with
$1\le n\le m$. A four-square polynomial triple in $K[X]$ with $\deg a=2n$ and $\deg b=\deg c=2m$ admits, after
possibly interchanging $b,c$ and choosing a square root $r$ of $ab+1$,
nonzero polynomials $P,Q,S,T\in K[X]$
such that
\begin{equation}\label{eq:two-largest-allocation}
a=PQ,\qquad b=ST,\qquad r-1=PS,\qquad r+1=QT,
\end{equation}
\begin{equation}\label{eq:two-largest-system}
\begin{gathered}
QT-PS=2,\qquad c=(P+T)(Q+S),\\
PQST(P+T)(Q+S)+1=U^2
\end{gathered}
\end{equation}
for some $U\in K[X]$. Their degrees have the form
\begin{equation}\label{eq:two-largest-degrees}
(\deg P,\deg Q,\deg S,\deg T)
=(i,2n-i,n+m-i,m-n+i),\qquad 0\le i\le n.
\end{equation}
If $n<m$, then
\[
\deg(c-b)=n+m,
\]
and both $b$ and $c$ admit a factorization into relatively prime
nonconstant factors of degrees $j,2m-j$, where
$m-n\le j\le m$.
\end{proposition}

\begin{proof}
By \Cref{lem:regularity}, choose the sign of the square root $r$ of 
$ab+1$ so that $c=a+b+2r$.
Since $(r+1)-(r-1)=2$ is a unit in $K[X]$, the factors
$r-1,r+1$ are relatively prime. Put
\[
P=\gcd(a,r-1),\qquad Q=a/P,\qquad
S=(r-1)/P,\qquad T=(r+1)/Q.
\]
The identity $(r-1)(r+1)=ab$ implies that every complete
irreducible-power factor of $a$ divides exactly one of $r-1,r+1$.
Thus $Q\mid r+1$. All four expressions are
polynomials, and multiplication gives $ST=b$.
The remaining identities follow from
\[
QT-PS=2,\qquad
(P+T)(Q+S)=PQ+ST+PS+QT=a+b+2r.
\]
The last identity in \eqref{eq:two-largest-system} is the fourth
square condition.

Initially $0\le i=\deg P\le2n$. Since $\deg r=n+m$,
degree comparison gives the four degrees in
\eqref{eq:two-largest-degrees}. We may interchange $b,c$
because the four-square conditions are symmetric in the entries.
For the permuted triple take $\widetilde r=-a-r$ as the new
square root of $ac+1$. This preserves the chosen regularity sign:
\[
ac+1=\widetilde r^{\,2},\qquad b=a+c+2\widetilde r.
\]
Moreover, $\widetilde r-1\equiv-(r+1)\pmod a$. Thus this
operation exchanges the two allocations of the factors of $a$,
replacing $i$ by $2n-i$. We may therefore assume $i\le n$.

If $n<m$, then $\deg r>2n$, whence
$\deg(c-b)=\deg(a+2r)=n+m$. Also
\[
\deg T=m-n+i>\deg P,\qquad
\deg S=n+m-i>\deg Q.
\]
Consequently, the factors $P+T,Q+S$ of $c$ have the same
degrees as $T,S$, respectively. A common divisor of $S,T$
divides $QT-PS=2$. A common divisor of $P+T,Q+S$ also
divides $QT-PS$, by reduction modulo that divisor.
Both factorizations are therefore relatively prime, with
$j=m-n+i$.
\end{proof}

\begin{corollary}\label{cor:two-largest-reducible}
Under the hypotheses of \Cref{prop:two-largest-split}, if $n<m$,
each of $b,c$ has at least two distinct irreducible
factors over $K$. 
\end{corollary}

\begin{proof}
Apply the relatively prime nonconstant factorizations in
\Cref{prop:two-largest-split}.
\end{proof}

\subsection{Equal degrees: composition obstructions and factorization}
\label{sec:equal-degrees}
 
We refine the factor-allocation system for equal degrees. The first
step is a composition obstruction, which will exclude the rational
unsplit case. We then compare the real construction and derive the
remaining rational factor-degree profiles.

\begin{lemma}\label{lem:odd-degree-composition}
Let \(K\in\{\Q,\R\}\). Let \(A\in K[X]\) be nonconstant and let
 \(F\in K[Y]\setminus\{0\}\). If \(F(A)\) is a square in \(K[X]\), then, over an
algebraic closure \(\overline K\), the polynomial \(F\) has at most one
distinct root of odd multiplicity.

In particular, if \(\deg F\ge3\) is odd and \(F(A)\) is a square in
\(K[X]\), then \(F\) has a repeated root; equivalently,
\begin{equation}\label{eq:disc-zero}
 \disc(F)=0.
\end{equation}
\end{lemma}

\begin{proof}
Factor \(F\) over \(\overline K\) as
\[
 F(Y)=\gamma\prod_{i=1}^{s}(Y-\rho_i)^{e_i},
 \qquad \gamma
=\lc(F)
\in K^{*},
\]
where \(\rho_1,\ldots,\rho_s\) are pairwise distinct and \(e_i\ge1\).
Then $F(A)=\gamma\prod_{i=1}^{s}(A-\rho_i)^{e_i}.$ For \(i\ne j\), the polynomials \(A-\rho_i\) and \(A-\rho_j\) are
relatively prime in \(\overline K[X]\), because their difference is the
nonzero constant \(\rho_j-\rho_i\).

Suppose that \(e_i\) is odd. Since the polynomials
\(A-\rho_1,\ldots,A-\rho_s\) are pairwise relatively prime and \(F(A)\)
is a square, every irreducible factor of \(A-\rho_i\) occurs in
\(A-\rho_i\) with even multiplicity. Hence,
\[
 A-\rho_i=\alpha_i B_i^2,
\]
for some \(\alpha_i\in\overline K^*\) and
\(B_i\in\overline K[X]\). Since \(\overline K\) is algebraically closed,
we may absorb a square root of \(\alpha_i\) into \(B_i\) and write
\[
 A-\rho_i=C_i^2.
\]
 If two distinct roots \(\rho_i\) and \(\rho_j\) both had odd
multiplicity, then
$A-\rho_i=C_i^2$ and $A-\rho_j=C_j^2$
for some \(C_i,C_j\in\overline K[X]\). Subtracting gives
\begin{equation}\label{eq:difference-two-squares}
 (C_i-C_j)(C_i+C_j)=\rho_j-\rho_i\in\overline K^*.
\end{equation}
Both factors on the left are therefore constant. Since
\(\operatorname{char}K=0\), it follows that \(C_i\) and \(C_j\) are
constant, and hence \(A\) is constant, a contradiction. Thus at most one distinct root
of \(F\) has odd multiplicity.

If \(\deg F\ge3\) is odd and \(F\) were squarefree, all of its roots
would have multiplicity one, so at least three distinct roots would have
odd multiplicity. This contradicts the first part. Hence, \(F\) has a
repeated root and \eqref{eq:disc-zero} follows.
\end{proof}

\begin{remark}\label{rem:odd-composition-structure}
Under the hypotheses of \Cref{lem:odd-degree-composition}, if
\(F\) has odd degree, then over \(\overline K\) one necessarily has
\[
 F(Y)=\gamma(Y-\rho)G(Y)^2,
\]
for suitable \(\gamma\in\overline K^*\), \(\rho\in\overline K\), and
\(G\in\overline K[Y]\).
\end{remark}

\begin{corollary}\label{cor:odd-composition}
Let \(K\in\{\Q,\R\}\). Let \(A\in K[X]\) be nonconstant and let
\(F\in K[Y]\) be squarefree of odd degree at least three. Then
\(F(A)\) is not a square in \(K[X]\).
\end{corollary}

\begin{proof}
If \(F(A)\) were a square, \Cref{lem:odd-degree-composition} would imply that \(F\) has
a repeated root over \(\overline K\), contrary to squarefreeness.
\end{proof}

\begin{remark}\label{rem:linear-exception}The hypothesis \(\deg F\ge3\) in \Cref{cor:odd-composition} is essential. For example,
\(F(Y)=Y\) and \(A(X)=X^2\) give \(F(A)=X^2\), although \(F\) has no
repeated root.
\end{remark}

We now apply the composition obstruction to equal degrees
\((2n,2n,2n)\), with \(n\ge1\). Setting \(m=n\) in
\Cref{prop:two-largest-split} gives
\[
 (\deg P,\deg Q,\deg S,\deg T)
 =(i,2n-i,2n-i,i),\qquad 0\le i\le n.
\]
The profile $i=0$ in \Cref{prop:two-largest-split} has $P$
constant and hence $a\mid r+1$. More generally, an allocation is
unsplit at an entry if that entry divides one of the two sign
factors. The following proposition excludes either possibility
over $\Q$ when the entry degrees are equal.

\begin{proposition}\label{prop:unsplit}
Let \(n\) be a positive integer. Suppose that \(a,b,c\in\Q[X]\) satisfy \eqref{eq:four-squares} and
\[
 \deg a=\deg b=\deg c=2n.
\]
Then \(a\nmid r-1\) and \(a\nmid r+1\), and the analogous
conclusion holds for either entry in each of the three pairwise square
identities. Equivalently, the complete irreducible-power factors of
an entry do not all occur in the same sign factor of a pairwise square root.
\end{proposition}

\begin{proof}
Suppose, on the contrary, that \(a\mid r-1\) or \(a\mid r+1\).
Since \(\deg a=\deg b=\deg c=2n>0\), \Cref{lem:regularity} implies that
the triple is regular. If initially \(a\mid r+1\), replace
\(r\) by \(-r\); this preserves \(r^2=ab+1\) and gives
\(a\mid r-1\). We retain this divisibility choice of sign and
therefore allow either regularity sign:
\begin{equation}\label{eq:regularity-epsilon}
 c=a+b+2\varepsilon r,
 \qquad \varepsilon\in\{1,-1\}.
\end{equation}
Write
\begin{equation}\label{eq:r-ak-plus-one}
 r=ak+1.
\end{equation}
Since \(\deg r=2n=\deg a\), the quotient \(k\) is a rational constant.
The identity \(ab+1=r^2\) and \eqref{eq:regularity-epsilon} give
\begin{equation}\label{eq:epsilon-family}
 b=ak^2+2k,
 \qquad
 c=a(k+\varepsilon)^2+2(k+\varepsilon).
\end{equation}
If \(\varepsilon=-1\), interchange \(b,c\), which preserves
all four square conditions, and put \(j=k-1\). This is a
relabeling of the entries and the constant parameter, with no
change of variable. Then \eqref{eq:epsilon-family} becomes
$b=aj^2+2j,$ $c=a(j+1)^2+2(j+1).$
Thus, after this relabeling and renaming \(j\) as \(k\), both sign
cases reduce to
\begin{equation}\label{eq:consecutive-family}
 b=ak^2+2k,
 \qquad
 c=a(k+1)^2+2(k+1).
\end{equation}
Here \(k\ne0,-1\), because these values would make \(b=0\) or
\(c=0\), respectively. It follows that
\begin{equation}\label{eq:Fk-composition}
 abc+1=F_k(a),
\end{equation}
where
\begin{equation}\label{eq:Fk-definition}
 F_k(Y)=Y(Yk^2+2k)(Y(k+1)^2+2(k+1))+1.
\end{equation}
Expanding the standard cubic discriminant gives
\begin{equation}\label{eq:Fk-discriminant}
 \disc_Y(F_k)=k^3(k+1)^3h(k),
 \qquad h(k)=32k^3+85k^2-11k-32.
\end{equation}

For \(k\in\Q\setminus\{0,-1\}\), the first two factors in
\eqref{eq:Fk-discriminant} are nonzero. Moreover,
\[
 h(k)\equiv2k^3+k^2+k+1\pmod3.
\]
For \(k=0,1,2\), this is congruent to \(1,2,2\), respectively, so $h$ has no root in
\(\mathbb F_3\). Being cubic, it is irreducible over \(\mathbb F_3\). The polynomial \(h\) is primitive and its leading coefficient is
nonzero modulo \(3\), so the reduction criterion together
with Gauss's lemma \cite{Lang} implies that \(h\) is
irreducible over \(\Q\). In particular, \(h(k)\ne0\) for every
\(k\in\Q\). Hence, \(F_k\) is squarefree. But
\eqref{eq:four-squares} and \eqref{eq:Fk-composition} give
\(F_k(a)=u^2\), contradicting \Cref{cor:odd-composition}.
\end{proof}

\begin{corollary}\label{cor:equal-reducible}
Under the hypotheses of \Cref{prop:unsplit}, each of \(a,b,c\)
has at least two distinct irreducible factors over \(\Q\).
In particular, every entry is reducible in \(\Q[X]\).
\end{corollary}

\begin{proof}
Since \((r-1)(r+1)=ab\) and \(\gcd(r-1,r+1)=1\), every
complete irreducible-power factor of \(a\) divides exactly one of
\(r-1\) and \(r+1\). By \Cref{prop:unsplit}, these factors cannot
all divide the same one. Thus \(a\) has at least two distinct
irreducible factors. By symmetry, the same holds for \(b\) and \(c\).
\end{proof}

\begin{remark}\label{rem:real-unsplit}
The analogue of \Cref{prop:unsplit} is false over \(\R[X]\). The cubic
$ h(k)=32k^3+85k^2-11k-32$
has three distinct real roots, since
$\disc(h)=68\,574\,961>0.$
Moreover, \(h(0)\ne0\) and \(h(-1)\ne0\), so every root \(k\) of \(h\)
satisfies \(k\ne0,-1\).

Let \(k\) be such a root. By \eqref{eq:Fk-discriminant}, the cubic
\(F_k\) has a repeated root and hence can be written as
\[
 F_k(Y)=k^2(k+1)^2(Y-\rho)^2(Y-\sigma)
\]
for suitable \(\rho,\sigma\in\R\). Given any nonconstant \(P\in\R[X]\), put
\[
 a=\sigma+P^2,\qquad
 b=k^2a+2k,\qquad
 c=(k+1)^2a+2(k+1).
\]
Then
\[
 ab+1=(ka+1)^2,\qquad
 ac+1=((k+1)a+1)^2,
\]
\[
 bc+1=\bigl(k(k+1)a+2k+1\bigr)^2,
\]
and
\[
 abc+1=\bigl(k(k+1)(a-\rho)P\bigr)^2.
\]
Since \(k\ne0,-1\) and \(a\) is nonconstant, the polynomials $a$, $b$ and $c$ are
nonzero, nonconstant, and pairwise distinct. Moreover,
$ r=ka+1$
satisfies \(a\mid r-1\). Thus, these triples satisfy the unsplit
condition excluded over \(\Q[X]\) in \Cref{prop:unsplit}.

This is a composition construction: the quadratic triple obtained by
taking \(P=X\) is composed with the polynomial \(P\). It is the same
mechanism used in \cite[Proposition~4.1]{Jurasic2026}.
\end{remark}
\medskip
The following corollary follows from the same composition
construction as in Jurasi\'{c} \cite[Corollary~4.2]{Jurasic2026}.

\begin{corollary}\label{cor:real-equal-degrees}
For every positive integer \(n\), there are infinitely many four-square
polynomial triples over \(\R[X]\) with degree pattern
\((2n,2n,2n)\), with $n\geq 1$.
\end{corollary}

\begin{proof}
In \Cref{rem:real-unsplit}, choose any polynomial \(P\in\R[X]\) of
degree \(n\). Then \(\deg a=\deg b=\deg c=2n\). For instance, taking
$P_t=X^n+t,$ $t\in\R$,
gives infinitely many distinct polynomials \(a=\sigma+P_t^2\), and
hence infinitely many distinct triples.
\end{proof}

Returning to triples over \(\Q[X]\), we now specialize the general
factor-allocation theorem to the equal-degree pattern \((2n,2n,2n)\), with $n\geq 1$.

\begin{proposition}\label{prop:equal-degree-split}
Let \(n\) be a positive integer. After a permutation of the entries and
a choice of the sign of \(r\), any four-square polynomial triple over \(\Q[X]\) with degree
pattern \((2n,2n,2n)\), with $n\geq 1$, gives nonconstant polynomials
\(P,Q,S,T\in\Q[X]\) such that
\begin{equation}\label{eq:det-system}
 r-1=PS,
 \qquad r+1=QT,
 \qquad a=PQ,
 \qquad b=ST,
 \qquad QT-PS=2,
\end{equation}
\begin{equation}\label{eq:c-factor}
 c=(P+T)(Q+S),
\end{equation}
and
\begin{equation}\label{eq:quartic-fourth}
 PQST(P+T)(Q+S)+1=U^2
\end{equation}
for some \(U\in\Q[X]\). Up to interchanging \(a\) and \(b\), there is
an integer \(i\) with \(1\le i\le n\) such that
\begin{equation}\label{eq:general-factor-degrees}
 (\deg P,\deg Q,\deg S,\deg T)
 =(i,2n-i,2n-i,i).
\end{equation}
Thus the split-factor system has \(n\) candidate degree profiles up to
this symmetry.

Conversely, suppose nonconstant \(P,Q,S,T,U\in\Q[X]\) satisfy
\(QT-PS=2\) and \eqref{eq:quartic-fourth}. If
\[
 a=PQ,
 \qquad b=ST,
 \qquad c=(P+T)(Q+S)
\]
are pairwise distinct and \(c\) is nonzero and nonconstant, then they
form a four-square polynomial triple.
\end{proposition}

\begin{proof}
Apply \Cref{prop:two-largest-split} with $m=n$.
It gives \eqref{eq:det-system}--\eqref{eq:quartic-fourth} and
\[
(\deg P,\deg Q,\deg S,\deg T)
=(i,2n-i,2n-i,i),\qquad 0\le i\le n.
\]
If $i=0$, then $P$ is constant. The identities $a=PQ$ and
$r+1=QT$ imply $a\mid r+1$, contrary to \Cref{prop:unsplit}.
Therefore $1\le i\le n$, so all four factors are nonconstant.
The symmetry stated in the proposition is also immediate:
interchanging $a,b$ sends
$(P,Q,S,T)$ to $(S,T,P,Q)$, preserves the system, and replaces
$i$ by $2n-i$.

For the converse, set
$ r=PS+1=QT-1.$
Then,
\begin{align*}
 ab+1&=PQST+1=(PS+1)^2=r^2,\\
 ac+1&=PQ(P+T)(Q+S)+1=(a+r)^2,\\
 bc+1&=ST(P+T)(Q+S)+1=(b+r)^2.
\end{align*}
The identity \eqref{eq:quartic-fourth} gives \(abc+1=U^2\). The stated
nonzero, nonconstant, and distinctness assumptions therefore give a
four-square polynomial triple.
\end{proof}

\begin{remark}
\label{rem:low-degree-profiles}
For \(n=1\), \Cref{prop:equal-degree-split} yields only the
factor-degree profile
\[
 (\deg P,\deg Q,\deg S,\deg T)=(1,1,1,1).
\]
This profile is only a necessary degree-theoretic condition.
The corresponding degree-two problem is solved completely in Section \ref{sec3},
where the real triples are classified and the existence of rational
quadratic triples is excluded.
\end{remark}

\subsection{Application to a regular parametrized family}
\label{sec:regular-family}

As a final application, we consider a regular parametrized family
of polynomial Diophantine triples and give   sufficient conditions under which the fourth expression in
the four-square conditions is not a square.

\begin{proposition}\label{prop:regular-family}
Let $p\in\Q[X]$ be nonconstant and
$q\in\Q[X]\setminus\{0,-1\}$. Define
\begin{equation}\label{eq:regular-family}
a=p,\qquad b=q(pq+2),\qquad
c=(q+1)(p(q+1)+2).
\end{equation}
Then $\{a,b,c\}$ is a regular polynomial Diophantine triple of distinct
nonconstant entries. If $N=\deg p$ and $d=\deg q$, its degree pattern is
$(N,N+2d,N+2d).$

If $q$ is constant or $p$ is a polynomial square, then $abc+1$
is not a square in $\Q[X]$.
\end{proposition}

\begin{proof}
Set
\[
r=pq+1,\qquad H=pq(q+1)+2q+1.
\]
Direct multiplication gives
\begin{equation}\label{eq:regular-three-squares}
ab+1=(pq+1)^2,\qquad
ac+1=(p(q+1)+1)^2,\qquad bc+1=H^2,
\end{equation}
and $c=a+b+2r$. Also, since $a=p$ and $bc=H^2-1$,
\begin{equation}\label{eq:regular-fourth}
 abc+1=pH^2+1-p.
\end{equation}

Since $p$ is nonconstant and $q\ne0,-1$, all three entries are
nonzero and nonconstant. Their differences are
\[
b-a=(q^2-1)p+2q,\qquad
c-a=q(q+2)p+2(q+1),\qquad
c-b=(2q+1)p+2.
\]
For nonconstant $q$, comparison of degrees shows that none is zero.
For constant $q$, a vanishing coefficient of $p$ can occur only
at $q=1$ in the first difference, $q=-2$ in the second, and
$q=-1/2$ in the third, after excluding $q=0,-1$.
The corresponding differences are $2,-2,2$, respectively.
Thus the entries are pairwise distinct.
Comparison of leading terms gives the displayed degree pattern,
both for $d=0$ and for $d>0$.

If $q$ is constant and $abc+1$ were a square,
\Cref{lem:degrees} would give the degree pattern $(2n,2n,2n)$
for some $n\ge1$. But $r-1=pq=aq$, so $a\mid r-1$,
contrary to \Cref{prop:unsplit}.

Now suppose $p=w^2$ for a nonconstant $w\in\Q[X]$.
If $abc+1=u^2$, then \eqref{eq:regular-fourth} gives
\[
 (u-wH)(u+wH)=1-w^2.
\]
Here $\deg H=N+2d$, so $\deg(wH)=3N/2+2d>N$.
The squares of $u$ and $wH$ have the same leading term.
Choosing the sign of $u$ so that $\lc(u)=\lc(wH)$ therefore
gives $\deg(u+wH)=3N/2+2d>N$, whereas the nonzero polynomial
$1-w^2$ has degree $N$. This contradicts the factorization.
\end{proof}

\section*{Acknowledgements}

The authors were supported by the Croatian Science Foundation under the project No. IP-2022-10-5008. A. J. was also supported by the European Union -- NextGenerationEU,
project number uniri-iz-25-62-ALGEBRA. During the preparation of this manuscript, the authors used ChatGPT-5.6 to investigate generalizations, enhance text readability, and assist with LaTeX formatting. After using this tool, the authors reviewed and edited the content and take full responsibility for the content of the publication.

\normalsize


\begin{thebibliography}{99}


\bibitem{Dujella2001}
A.~Dujella,
\newblock \emph{Diophantine \(m\)-tuples and elliptic curves},
\newblock J. Th\'eor. Nombres Bordeaux \textbf{13} (2001), no.~1,
111--124,
\newblock \url{https://doi.org/10.5802/jtnb.308}.

\bibitem{DujellaBook2024}
A.~Dujella,
\newblock \emph{Diophantine \(m\)-tuples and Elliptic Curves},
\newblock Developments in Mathematics, vol.~79,
Springer, Cham, 2024,
\newblock \url{https://doi.org/10.1007/978-3-031-56724-7}.

\bibitem{DujellaFuchs2003}
A.~Dujella and C.~Fuchs,
\newblock \emph{A polynomial variant of a problem of Diophantus and Euler},
\newblock Rocky Mountain J. Math. \textbf{33} (2003), no.~3,
797--811,
\newblock \url{https://doi.org/10.1216/rmjm/1181069929}.

\bibitem{DujellaFuchs2004}
A.~Dujella and C.~Fuchs,
\newblock \emph{Complete solution of the polynomial version of a problem
of Diophantus},
\newblock J. Number Theory \textbf{106} (2004), no.~2, 326--344,
\newblock \url{https://doi.org/10.1016/j.jnt.2003.12.011}.

\bibitem{DujellaFuchsTichy2002}
A.~Dujella, C.~Fuchs, and R.~F.~Tichy,
\newblock \emph{Diophantine \(m\)-tuples for linear polynomials},
\newblock Period. Math. Hungar. \textbf{45} (2002), no.~1--2, 21--33,
\newblock \url{https://doi.org/10.1023/A:1022389711839}.

\bibitem{DujellaFuchsWalsh2006}
A.~Dujella, C.~Fuchs, and P.~G.~Walsh,
\newblock \emph{Diophantine \(m\)-tuples for linear polynomials.
II\@. Equal degrees},
\newblock J. Number Theory \textbf{120} (2006), no.~2, 213--228,
\newblock \url{https://doi.org/10.1016/j.jnt.2005.12.005}.

\bibitem{DujellaJurasic2010}
A.~Dujella and A.~Jurasi\'c,
\newblock \emph{On the size of sets in a polynomial variant of a problem
of Diophantus},
\newblock Int. J. Number Theory \textbf{6} (2010), no.~7, 1449--1471,
\newblock \url{https://doi.org/10.1142/S1793042110003575}.

\bibitem{DujellaJurasic2011}
A.~Dujella and A.~Jurasi\'c,
\newblock \emph{Some Diophantine triples and quadruples for quadratic
polynomials},
\newblock J. Comb. Number Theory \textbf{3} (2011), no.~2, 123--141.

\bibitem{DujellaSzalay}
A.~Dujella and L.~Szalay,
\newblock \emph{Four squares from three numbers},
\newblock arXiv:2506.14013 [math.NT], 2025,
\newblock \url{https://arxiv.org/abs/2506.14013}.


\bibitem{Jones1976}
B.~W.~Jones,
\newblock \emph{A variation on a problem of Davenport and Diophantus},
\newblock Q. J. Math. \textbf{27} (1976), no.~3, 349--353,
\newblock \url{https://doi.org/10.1093/qmath/27.3.349}.

\bibitem{Jurasic2011}
A.~Jurasi\'c,
\newblock \emph{Diophantine \(m\)-tuples for quadratic polynomials},
\newblock Glas. Mat. Ser. III \textbf{46(66)} (2011), no.~2, 283--309,
\newblock \url{https://doi.org/10.3336/gm.46.2.02}.

\bibitem{Jurasic2026}
A.~Jurasi\'c,
\newblock \emph{Four squares from three real polynomials},
\newblock Rad HAZU, to appear. 
\newblock \url{https://arxiv.org/abs/2607.06586}.

\bibitem{Lang}
S.~Lang,
\newblock \emph{Algebra},
\newblock revised third edition,
Graduate Texts in Mathematics, vol.~211,
Springer-Verlag, New York, 2002,
\newblock \url{https://doi.org/10.1007/978-1-4613-0041-0}.

\bibitem{Mathematica}
Wolfram Research, Inc.,
\emph{Mathematica, Version~14.3},
Champaign, IL, 2025.
\url{https://www.wolfram.com/mathematica}.



\end{thebibliography}
\end{document}